\documentclass[journal,onecolumn,11pt]{IEEEtran}

\usepackage{pstricks}
\usepackage{graphics}
\usepackage{amssymb}
\usepackage{amsmath}
\usepackage{amsthm}
\usepackage{mathabx}
\usepackage{subfloat}

\newtheorem{thm}{Theorem}
\newtheorem{Lemma}{Lemma}

\theoremstyle{definition}
\newtheorem{definition}{Definition}
\newtheorem{example}{Example}

\newtheorem{note}{Note}

\newcommand{\ncom}{\newcommand}
\ncom{\etal}{\emph{et al.}}
\ncom{\beqn}{\begin{eqnarray*}} \ncom{\eeqn}{\end{eqnarray*}}
\ncom{\beq}{\begin{eqnarray}} \ncom{\eeq}{\end{eqnarray}}
\ncom{\be}{\begin{equation}} \ncom{\ee}{\end{equation}}
\ncom{\ben}{\begin{equation*}} \ncom{\een}{\end{equation*}}
\ncom{\nn}{\nonumber}
\ncom{\nl}{$n$-$L$}
\ncom{\nll}{$(n-1)$-$L$}
\ncom{\ml}{$\mathcal{L^{\mathrm{n}}R}$\hspace{2pt}}
\ncom{\mr}{$\mathcal{LR^{\mathrm{n}}}$\hspace{2pt}}
\ncom{\oln}{$\mathcal{O_{L^{\mathrm{n}}R}}$\hspace{2pt}}
\ncom{\orn}{$\mathcal{O_{LR^{\mathrm{n}}}}$\hspace{2pt}}
\ncom{\mo}{$\mathcal{O_{\sigma}}$\hspace{2pt}}
\ncom{\ps}{$\mathcal{P_{\sigma}}$\hspace{2pt}}
\ncom{\cL}{$\mathcal{L}$\hspace{2pt}}
\ncom{\cR}{$\mathcal{R}$\hspace{2pt}}
\ncom{\cLs}{$\mathcal{L}s$\hspace{2pt}}
\ncom{\cRs}{$\mathcal{R}s$\hspace{2pt}}

\begin{document}

\title{Analysis of Stable Periodic Orbits in 1-D Linear Piecewise Smooth Maps}

\author{Bhooshan Rajpathak,
        Harish K. Pillai, Santanu Bandopadhyay 
\thanks{Bhooshan Rajpathak is with the Dept of Electrical Engg,
IIT Bombay, Mumbai 400076, India, e-mail: bhooshan@ee.iitb.ac.in}%
\thanks{Harish K. Pillai is with the Dept of Electrical Engg,
IIT Bombay, Mumbai 400076, India, e-mail: hp@ee.iitb.ac.in}%
\thanks{Santanu Bandopadhyay is with the Dept of Energy Sc and Engg,
IIT Bombay, Mumbai 400076, India, e-mail: santanu@me.iitb.ac.in}%
}

\maketitle

\begin{abstract}
%

By varying a parameter of a one-dimensional piecewise smooth map, stable periodic orbits are observed. In this paper, complete analytic characterization of these stable periodic orbits is obtained. An interesting relationship between the cardinality of orbits and their period is established. It is proved that for any $n$, there exist $\phi(n)$ distinct admissible patterns of cardinality $n$. An algorithm to obtain these distinct admissible patterns is outlined. Additionally, a novel algorithm to find the range of parameter for which the orbit exists is proposed.  
\end{abstract}


\IEEEpeerreviewmaketitle

\section{Introduction}
\IEEEPARstart{D}{ynamics} of piecewise-smooth systems are encountered in various applications in electrical engineering and physics: controlled buck converter \cite{deane}, boost converter in discontinuous mode \cite{tse}, impact oscillators \cite{grebogi}, etc.  Significant theoretical understanding has been developed for systems with continuous maps.  Theory for piecewise-smooth maps has been partially developed in \cite{bernardo}.  Results related to the existence and stability of period-1 and period-2 fixed points in discontinuous maps have been reported in \cite{bernardo} and \cite{dutta}.  Analysis of bifurcation in piecewise-smooth systems has been shown in \cite{nusse} and \cite{yorke}. Most of the research efforts till date have been on analyzing piecewise-smooth systems through bifurcation diagrams and numerical simulation (e.g. \cite{1d}, \cite{border}, \cite{2d} and \cite{abed}). \cite{homer}, \cite{avrutina}, \cite{avrutinb} and \cite{avrutinc} have developed analytical studies to show the existence of higher periodic orbits. However, complete characterizations of stable periodic orbits for piecewise-smooth systems are yet to be developed. 
\par
In this paper we have examined the stable periodic orbits of piecewise-smooth systems analytically.  Such systems are often modeled as discrete maps which are divided into regions separated with borderlines. These maps are piecewise smooth and are differentiable everywhere except at the borderlines due to discontinuity. The one-dimensional piecewise smooth map that is investigated in this paper, is defined as \cite{border}:
\begin{equation}\label{basic}
x_{n+1}=f(x_n,a,b,\mu,l)=\left\{
                \begin{array}{lcl}
                ax_{n}+\mu & for & x_{n} \leq 0\\
                bx_{n}+\mu + l & for & x_{n} > 0
                \end{array}
        \right.
\end{equation}
\par
From the stability point of view, `$a$' and `$b$' are assumed to be in the range $(0,\,1)$. Height of the discontinuity is denoted by `$l$' and `$\mu$' is the parameter to be varied. Let us assume $l>0$ in equation \eqref{basic}. There are three cases as illustrated in Figure \ref{bifur}. 

\begin{subfigures}
\begin{figure}
\begin{center}
\scalebox{.5}{
\input{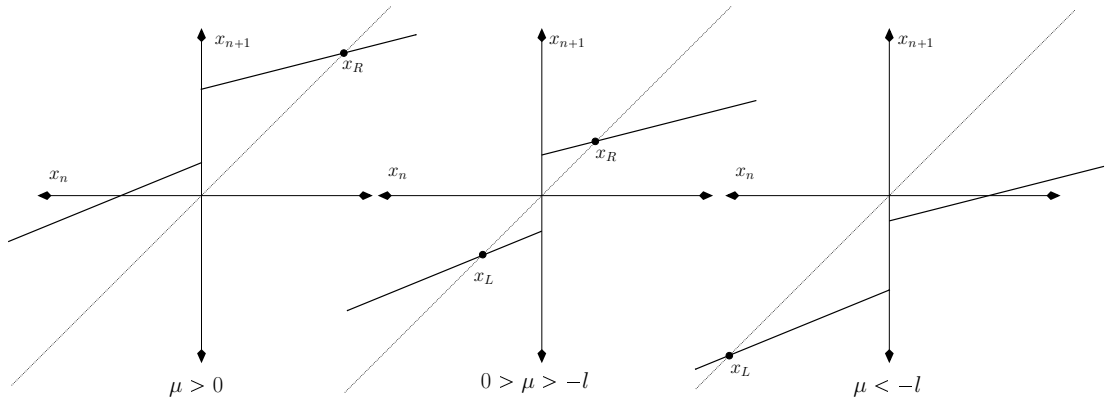}}
\caption{Graph of the map for $0<a<1$ and $0<b<1$, and $l>0$
\cite{border}} \label{bifur}
\end{center}
\end{figure}
\begin{enumerate}
\item[]{\textbf{Case 1:}} For $\mu>0$, there is a stable fixed point on the right-half plane. Location of the fixed point can be obtained from equation \eqref{basic} as $x_{R} =\frac{\mu+l}{1-b}$.
\item[]{\textbf{Case 2:}} For $0> \mu>-l$, there are two stable fixed points on both sides of the discontinuity as shown in Figure \ref{bifur}.
\item[]{\textbf{Case 3:}} For $\mu<-l$, there is a stable fixed point in the left half plane and it is given by $x_{L}=\frac{\mu}{1-a}$.
\end{enumerate}
\par
It may be observed that the left half of the map intersects the $45^{\circ}$ line for $\mu < 0$ and the right half of the map intersects this line for $\mu > -l$. This implies that the fixed point $x_{L}$ collides with the border at $\mu = 0$ and the fixed point $x_{R}$ collides with the border at $\mu =-l$. Therefore two border collision events are expected as $\mu$ is varied.
\par
Three additional cases may be observed when $l<0$ (see Figure \ref{bifur1}).
\begin{enumerate}
\item[]{\textbf{Case 4:}} For $\mu<0$, there is a stable fixed point in the left half plane and it is given by $x_{L}=\frac{\mu}{1-a}$.
\item[]{\textbf{Case 5:}} For $-l>\mu>0$, there is no fixed point. 
\item[]{\textbf{Case 6:}} For $\mu >-l$, there is another stable fixed point in the right half plane: $x_{R} =\frac{\mu+l}{1-b}$.
\end{enumerate}
\begin{figure}
\begin{center}
\scalebox{.5}{
\input{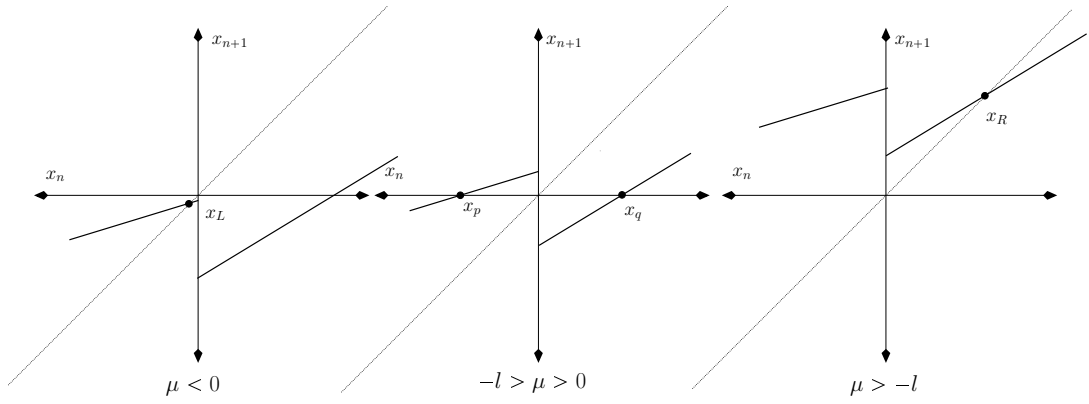}}
\caption{Graph of the map for $0<a<1$ and $0<b<1$, and $-l>0$}
\label{bifur1}
\end{center}
\end{figure}
\end{subfigures}
Case 5 is the most interesting amongst the six listed above as it contains no fixed point. This case has been analyzed in detail in this paper. As $l$ is independent of $\mu$, without loss of generality, it may be assumed that $l=-1$. One can see that when $x_{n}\in(x_{p}, 0]$ (which is in the closed left half plane), $x_{n+1}$ belongs to the right half plane. As the system is stable, it can be concluded that after some $k$ steps, the point $x_{n+k}$ returns to the left half plane again. This leads us to the following questions:
\begin{itemize}
\item Do periodic orbits exist for such systems?
\item If yes, then how to characterize them?
\end{itemize}

These questions are answered in this paper. It may be noted that the range of $\mu$ is crucial to determine the existence of orbits. Only when $0<\mu<-l$, is there a possibility of existence of orbits whereas for all other ranges of $\mu$, only fixed points exist. This motivates one to find the range of $\mu$ for the existence of certain specific kind of orbits of prescribed periodicity. It is shown in this paper, that a complete characterization of all orbits based on the range of $\mu$ is possible. 
\par

\pagebreak
\section{Preliminaries}
First we define the term periodic orbit \cite{aligood}.
\begin{definition}\label{periodic}
Let $f$ be a map from $\mathbb{R}$ to $\mathbb{R}$. We call $p$ a periodic point of period $k$ if $f^k(p)=p$, where $k$ is the smallest such positive integer. The orbit with initial point $p$ (which consists of $k$ points) is called a periodic orbit of period k. We will often use the abbreviated terms period-$k$ point and period-$k$ orbit for a periodic orbit having period-$k$.
\end{definition}

Let $\mathcal{L}=(-\infty,\,0]$ (the left half plane) and $\mathcal{R}=(0,\,\infty)$ (the right half plane). Given a particular sequence of points $\{x_n\}_{n\geq 0}$ through which the system evolves, one can convert (code) this sequence into a sequence of \cLs and \cRs by indicating which of the two sets ($\mathcal{L}$ or $\mathcal{R}$) the corresponding point belongs to. Clearly, a periodic orbit has a string of \cLs and \cRs that keeps repeating. We call this repeating string, a \emph{pattern} and denote it by $\sigma$. The length of the string $\sigma$ is denoted by $|\sigma |$ and gives the number of symbols in the pattern i.e., the period of the orbit. A periodic orbit with a pattern $\sigma$ is denoted as $\mathcal{O_{\sigma}}$. \ps denotes the interval of parameter $\mu$ for which orbit \mo exists. The sum of geometric series $1+k+k^2+\ldots+k^n$ is denoted by $S_n^k$.

\begin{definition}
A periodic orbit \mo is termed as admissible if \ps$\neq\phi$. The pattern of an admissible orbit is called an admissible pattern.
\end{definition}
\begin{definition}
If a pattern of a periodic orbit \mo consists of only one \cR and multiple \cLs or vice-versa, it is called an atomic pattern.
\end{definition}
Thus, there are two types of atomic patterns; those with pattern $\overbrace{\mathcal{LLL\cdots\cdots LL}}^n\mathcal{R}$, abbreviated as \ml (termed as \cL-atomic pattern) and those with pattern $\mathcal{L}\overbrace{\mathcal{RRR\cdots\cdots RR}}^n$, abbreviated as \mr (termed as \cR-atomic pattern). The pattern $\mathcal{LR}$ is both, \cL-atomic as well as \cR-atomic.
\begin{definition}
A pattern is called a molecular pattern if it is made up of a combination of atomic patterns.
\end{definition}
\begin{example}
$\mathcal{LLRLLRLR}$ is a molecular pattern. It is made up by combining the atomic patterns $\mathcal{LLR}$ and $\mathcal{LR}$.
\end{example}

\section{Analysis of Periodic Orbits}

\begin{Lemma}[Atomic Lemma]\label{atomic}
An atomic pattern of any period is admissible.
\end{Lemma}
\begin{IEEEproof}
Consider an atomic orbit \oln with period $n+1$. We write down the inequalities as: 
\begin{align*}
x_{0}&\leq0 ,\\
x_1&=ax_0+\mu\leq0 ,\\
x_2&=ax_1+\mu\leq0,\\
&=a^2x_0+(a+1)\mu\leq0,\\
\vdots\nn\\
x_{n-1}&=a^{n-1}x_0+\mu S_{n-2}^a\leq 0,\\
x_n&=a^{n}x_0+\mu S_{n-1}^a>0,\\
x_{n+1}&=x_{0}=bx_n+\mu-1\leq0,\\
\therefore\, x_0 &=\frac{(a^{n-1}b+a^{n-2}b+\ldots+ab+b+1)\mu-1}{1-a^{n}b}.
\end{align*}
\par
Substituting the value of $x_0$ into the list of inequalities above, would yield a list of upper bounds for $\mu$ (whenever the point $x_i$ is in \cL) and lower bounds for $\mu$ (when the point $x_i$ is in \cR). We denote upper bounds by $\mu^{upper}_i$ and lower bounds by $\mu^{lower}_i$.
We define
$\mu_{2}=\underset{i}{min}(\mu^{upper}_{i})$ and  $\mu_{1}=\underset{i}{max}(\mu^{lower}_i)$. Therefore, \ps$=(\mu_{1},\,\mu_{2}]$. A simple algebraic manipulation of the inequalities above gives:
\ben
\mathcal{P_{L^{\mathrm{n}}R}}=\mathbf{\left(\frac{a^{n}}{S_n^a},\quad \frac{a^{n-1}}{a^{n-1}b+S_{n-1}^a}\right]}.
\een
Let us assume $\mathcal{P_{L^{\mathrm{n}}R}}=\phi$.
\begin{align*}
\therefore\, & \frac{a^{n}}{S_n^a}>\frac{a^{n-1}}{a^{n-1}b+S_{n-1}^a}.\\
\therefore\, & a^n\times \bigg(a^{n-1}b+S_{n-1}^a\bigg)-a^{n-1}\times \bigg(S_n^a\bigg)>0.\\
\therefore\, & a^{n-1}\bigg[a^nb+aS_{n-1}^a-S_n^a\bigg]>0.\\
\therefore\, & -a^{n-1}(1-a^nb)>0.
\end{align*}
which is a contradiction as $a,b\in(0,\, 1)$. Hence $\mathcal{P_{L^{\mathrm{n}}R}}\neq\phi$.
\par
Similarly, consider an atomic orbit \orn. We write down the inequalities as: 
\begin{align*}
x_{0}&\leq0,\\
x_{1}&>0,\\
\vdots\\
x_{n+1}&=x_{0}\leq0.\\
\therefore\,x_{0}&=\frac{(b^{n-1}+b^{n-2}+\ldots+b+1)(\mu-1)+b^n\mu}{1-b^{n}a}
\end{align*}
Finding $\mu_{1}$ and $\mu_{2}$ in the way as explained above, we get
\ben
\mathcal{P_{LR^{\mathrm{n}}}}=\mathbf{\left(\frac{ab^{n-1}+S_{n-2}^b}{ab^{n-1}+S_{n-1}^b},\quad \frac{S_{n-1}^b} {S_{n}^b}\right]}
\een
Further, it can be easily checked that $\mathcal{P_{LR^{\mathrm{n}}}}\neq\phi$.
\end{IEEEproof}

\begin{example}
Let us consider an orbit $\mathcal{O_{LR}}$. Here $x_{0}\leq0,\,x_{1}>0$ and $x_{2}=x_{0}$. From equation \eqref{basic} 
\begin{align}
x_{1}&=ax_{0}+\mu.\nn\\
x_{2}&=bx_{1}+\mu-1, \nn\\
&=abx_{0}+(b+1)\mu-1,  \nn\\
&=x_{0}.\nn\\
\therefore\,x_{0}&=\frac{(b+1)\mu-1}{1-ab}\leq0 \label{x0}.\\
\therefore\,\mu&\leq\frac{1}{b+1}.\nonumber
\end{align}
Substituting the value of $x_{0}$ in $x_{1}$ we get:
\begin{align}
x_{1}&=a\frac{(b+1)\mu-1}{1-ab}+\mu >0 \label{x1}.\\
\therefore\,\mu&>\frac{a}{1+a}.\nonumber
\end{align}
Hence $\mathcal{P_{LR}}=\left(\frac{a}{1+a},\,\frac{1}{1+b}\right]$.  
\par
For example, if we assume $a=\frac{1}{2}$ and $b=\frac{1}{3}$, then $\mathcal{P_{LR}}=\left(\frac{1}{3},\, \frac{3}{4}\right]$. Assume $\mu=\frac{3}{5}$. We substitute the values of $a,b,l$ and $\mu$ to find $x_{0},x_{1}$ and $x_{2}$. From equation \eqref{x0} and \eqref{x1}: 
\begin{align*}
x_{0}=x_{2}=\frac{-6}{25}.\\
x_{1}=\frac{12}{25}.
\end{align*}
Above analysis shows that orbit $\mathcal{O_{LR}}$ has one point in the closed left half plane and other point in the open right half plane and hence its pattern is $\sigma=\mathcal{LR}$. Moreover, $x_2=x_0$ shows that this is in fact a period-$2$ orbit. $\mathcal{P_{LR}}=\left(\frac{1}{3},\, \frac{3}{4}\right]$ gives the range for $\mu$  where the orbit $\mathcal{O_{LR}}$ is admissible.
\end{example}

\begin{note}
The map given by equation \eqref{basic} is invariant under transformation $f(x,a,b,\mu,l)\rightarrow f(-x,b,a,-[\mu+l],l)$. Due to replacement of $x$ by $-x$, involved patterns will be inverted (i.e., \cLs will become \cRs and vice-versa). Therefore, for the sake of simplicity, we will only consider \cL-atomic patterns. The results will be directly applicable to \cR-atomic patterns through the transformation mentioned above.
\end{note}

\subsection{Problem Formulation and Analysis}
We have proved that atomic orbits are admissible. Additionally, we have obtained a closed form solution for the range of $\mu$ for these atomic orbits. This leads to following questions:
\begin{enumerate}
\item Are atomic orbits the only kind of orbits? For example, can there be an orbit like $\mathcal{O_{LLLRR}}$?
\item Can we characterize all the possible types of admissible orbits?
\item For a given $n$, how many distinct patterns exist with period $n$?
\item Is there any algorithm to generate all the admissible patterns?

\end{enumerate}
\par
In this paper we provide answers to all the above questions. We take the first step towards characterization of all possible types of admissible patterns, by proving that certain combinations of \cLs and \cRs cannot appear in any admissible pattern $\sigma$.


\begin{Lemma}\label{thm:funda}
For any admissible orbit $\mathcal{O_{\sigma}}$, its pattern cannot contain consecutive \cLs and consecutive \cRs simultaneously.
\end{Lemma}

\begin{IEEEproof}
We know, $-l=1$ and $0<\mu<1$. We first find conditions on $\mu$ such that consecutive \cRs do not appear. Let us assume $x_{0}\leq0,x_1>0$ and $x_{2}\leq x_0$. Then from equation \eqref{basic}
\begin{align*}
x_{1}&=ax_{0}+\mu>0,\\
x_2=&abx_{0}+(b+1)\mu-1\leq x_0,\\
\therefore\,\mu\leq&\frac{(1-ab)x_0+1}{b+1}.
\end{align*}
substituting $x_0\leq 0$ in above equation we get, $\mu\leq \frac{1}{b+1}$.
\par
Now we find conditions on $\mu$ such that consecutive \cLs do not appear. Let us assume $x_{0}>0,x_1\leq 0$ and $x_{2}>x_0$. Then from equation \eqref{basic}
\begin{align*}
x_{1}&=bx_{0}+\mu-1\leq0,\\
x_2&=abx_0+(a+1)\mu-a>x_0.\\
\therefore\, \mu&>\frac{(1-ab)x_0+a}{a+1}.
\end{align*}
Substituting $x_0>0$ in above equation we get, $\mu>\frac{a}{a+1}$. Since $a,b\in(0,\,1)\Rightarrow\,\frac{a}{a+1}<\frac{1}{b+1}$. This proves the lemma.
\end{IEEEproof}

Summarizing, we can say:
\begin{itemize}
\item When $0<\mu\leq\frac{1}{b+1}$, then any \cR is always immediately followed by \cL. So in this range, patterns with consecutive \cRs do not exist (see Figure \ref{case3}).
\item When $\frac{a}{a+1}<\mu\leq1$, then any \cL is immediately followed by \cR. So in this range, patterns with consecutive \cLs do not exist (see Figure \ref{case2}).
\item When $\frac{a}{a+1}<\mu\leq\frac{1}{b+1}$, only possible pattern is $\mathcal{LR}$.
\end{itemize}
\begin{subfigures}
\begin{figure}[ht!]
\begin{center}
\scalebox{.7}{
\input{case3.pstricks}}
\caption{The range of `$\mu$' where only singleton \cR is possible.}
\label{case3}
\end{center}
\end{figure}
\begin{figure}[ht!]
\begin{center}
\scalebox{.7}{
\input{case2.pstricks}}
\caption{The range of `$\mu$' where only singleton \cL is possible.}
\label{case2}
\end{center}
\end{figure}
\end{subfigures}

This lemma helps us predict whether certain patterns are admissible or in-admissible e.g., $\mathcal{LLRR}$ and $\mathcal{LLRLRRLR}$ are clearly not admissible patterns as these patterns contain both consecutive \cLs and consecutive \cRs simultaneously. An important corollary of Lemma \ref{thm:funda} is that all the admissible patterns are either atomic patterns (\ml or \mr) or molecular patterns made up of purely \cL-atomic patterns or \cR-atomic patterns.
\par
 We generalize this lemma to find conditions on $\mu$ for \emph{at most} $n$ consecutive \cLs or \emph{at least} $n$ consecutive \cLs to appear in a pattern.

\begin{Lemma}[At Most \& At Least Lemma]
 When $\mu\leq\frac{a^{n-1}}{a^{n-1}b+s_{n-1}^a}$ then at least $n$ consecutive \cLs appear in the pattern and when $\mu>\frac{a^{n}}{S_n^a}$ then at most $n$ consecutive \cLs appear in the pattern.
\end{Lemma}

\begin{IEEEproof}
Let us assume $x_0\leq 0,x_1\leq 0, \ldots, x_{n-1}\leq 0, x_n>0$. From Atomic Lemma we get,
\begin{align}
 x_{n-1}&=a^{n-1}x_0+\mu S_{n-2}^a\leq 0, \nn\\
&\therefore\, x_0\leq-\frac{\mu S_{n-2}^a}{a^{n-1}}, \label{x01}\\
x_n&=a^nx_0+\mu S_{n-1}^a>0, \nn\\
&\therefore\, x_0>-\frac{\mu S_{n-1}^a}{a^n},\label{x02}\\
x_{n+1}&=ba^nx_0+\mu bS_{n-1}^a +\mu-1. \label{x03}
\end{align}
 First we find the condition on $\mu$ such that at least $n$ consecutive \cLs appear in a pattern. For this, we assume $x_{n+1}\leq x_0$. Then from equation \eqref{x03},  
\begin{equation}\label{uvalue}
 \mu\leq \frac{(1-ba^n)x_0+1}{(1+bS_{n-1}^a)}.
\end{equation}
Substituting \eqref{x01} in \eqref{uvalue} we get,
\begin{equation}
 \mathbf{\mu\leq \frac{a^{n-1}}{ba^{n-1}+S_{n-1}^a}}.
\end{equation}
\par
Now we find the condition on $\mu$ such that at most $n$ consecutive \cLs appear in a pattern. Let us assume $x_{n+1}> x_0$. Then from equation \eqref{x03},
\begin{equation}\label{uvalue2}
 \mu> \frac{(1-ba^n)x_0+1}{(1+bS_{n-1}^a)}.
\end{equation}
Substituting \eqref{x02} in \eqref{uvalue2} we get,
\begin{equation}
 \mathbf{\mu> \frac{a^{n}}{S_{n}^a}}.
\end{equation}
\end{IEEEproof}


At Most \& At Least Lemma (from now on we refer to it as AMAL Lemma) gives us the conditions on $\mu$ for the appearance of at most/at least $n$ consecutive \cLs in a pattern. In a similar way we can find the conditions on $\mu$ such that at most/at least $2n$ or $3n$ or $n-1$ consecutive \cLs appear in a pattern. Note that at least and at most conditions for consecutive \cRs can be found in the same fashion. It is important to note that all these conditions are placed on the parameter line $\mu$ in a specific order. We now prove that these conditions on $\mu$ are such that the admissible combinations for the molecular patterns are limited.


\begin{Lemma}
Every molecular pattern is a combination of at most two atomic patterns of successive cardinality.
\end{Lemma}

\begin{IEEEproof} From AMAL Lemma we get the at most/at least conditions on $\mu$. The statement of this lemma is equivalent to showing that on the $\mu$ parameter line (see Figure \ref{examp2}), at any given point, the two active conditions (one at least and one at most) come from succesive values of $n$. This is equivalent to showing that $\frac{a^{n}}{S_{n}^a}<\frac{a^{n-1}}{a^{n-1}b+S_{n-1}^a}<\frac{a^{n-1}}{S_{n-1}^a}$ for every $n\geq 2$. From Atomic Lemma we know that $\frac{a^{n}}{S_{n}^a}<\frac{a^{n-1}}{a^{n-1}b+S_{n-1}^a}$. Hence to prove this lemma it is enough to prove that $\frac{a^{n-1}}{a^{n-1}b+S_{n-1}^a}<\frac{a^{n-1}}{S_{n-1}^a}$. Let us assume 
\begin{align*}
 \frac{a^{n-1}}{a^{n-1}b+S_{n-1}^a}&>\frac{a^{n-1}}{S_{n-1}^a},\\
\therefore\, a^{n-1}S_{n-1}^a&>(a^{n-1}b+S_{n-1}^a)a^{n-1},\\
\therefore\, 0&>a^{2(n-1)}b.
\end{align*}
This is a contradiction as $a,b \in (0,\,1)$. Hence $\frac{a^{n-1}}{a^{n-1}b+S_{n-1}^a}<\frac{a^{n-1}}{S_{n-1}^a}$
\end{IEEEproof}

We summarize the above lemma into the following cases:
\begin{enumerate}
\item[]{\textbf{General Case:} $\mathbf{\mu\in\big(\frac{a^{n}}{S_{n}^a},\, \frac{a^{n-2}}{a^{n-2}b+S_{n-2}^a}\big]}$.} All the patterns consist of either $n-1$ or $n$ consecutive \cLs or a combination of both. 
\item[]{\textbf{Case a:} $\mathbf{\mu\in\big(\frac{a^{n}}{S_{n}^a},\, \frac{a^{n-1}}{a^{n-1}b+S_{n-1}^a}\big]}$.} The only possible pattern is the pattern having exactly $n$ consecutive \cLs and a single \cR. This is the pattern \ml. The above range of $\mu$ is nothing but the value of $\mathcal{P_{L^\mathrm{n}R}}$.
\item[]{\textbf{Case b:} $\mathbf{\mu\in\big(\frac{a^{n-1}}{a^{n-1}b+S_{n-1}^a},\,\frac{a^{n-1}}{S_{n-1}^a}\big]}$.} Only a combination of $n-1$ and $n$ consecutive \cLs \- appear in a pattern. Hence atomic orbits cannot exist here. We call this region as a \emph{molecular region}.
\item[]{\textbf{Case c:} $\mathbf{\mu\in\big(\frac{a^{n-1}}{S_{n-1}^a},\, \frac{a^{n-2}}{a^{n-2}b+S_{n-2}^a}\big]}$.} The only possible pattern is the pattern having exactly $n-1$ consecutive \cLs \- and a single \cR. This is the pattern of $\mathcal{L^{\mathrm{n-1}}R}$. The above range of $\mu$ is nothing but the value of $\mathcal{P_{L^\mathrm{n-1}R}}$.
\end{enumerate}
\begin{figure}[ht!]
\begin{center}
\scalebox{.7}{
\input{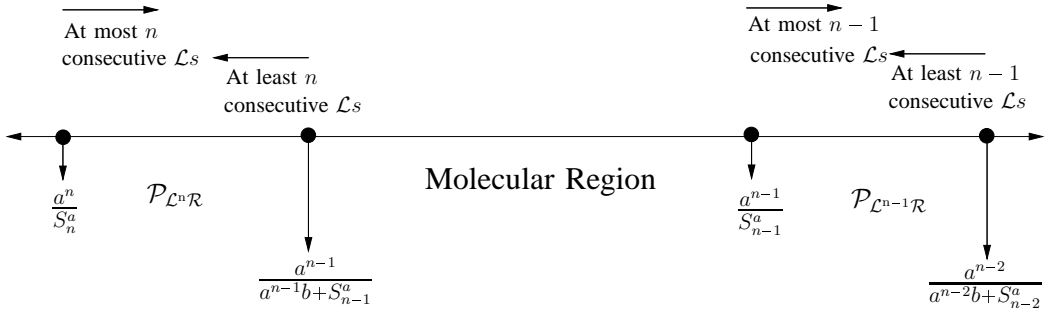}}
\caption{Position of cases with respect to range of parameter $\mu$}
\label{examp2}
\end{center}
\end{figure}
\begin{note} 
Consider an orbit $\mathcal{O_{L^{\mathrm{n}}RL^{\mathrm{n-1}}R}}$. Its pattern is a combination of patterns \ml and $\mathcal{L^{\mathrm{n-1}}R}$.  The range of $\mu$ for existence of $\mathcal{O_{L^{\mathrm{n}}RL^{\mathrm{n-1}}R}}$, $\mathcal{P_{L^{\mathrm{n}}RL^{\mathrm{n-1}}R}}$ can be calculated as 
\begin{equation*}
\mathbf{\mathcal{P_{L^{\mathrm{n}}RL^{\mathrm{n-1}}R}}=\Big(\frac{a^{2n-1}b+a^{n-1}}{a^{2n-1}b+(a^{n-1}b+1)S_{n-1}^a},\,\frac{a^{2n-2}b+a^{n-1}}{a^{2n-2}b^{2}+(a^{n-1}b+1)S_{n-1}^a}\Big]}.
\end{equation*}
Case b above tells us that $\mathcal{P_{L^{\mathrm{n}}RL^{\mathrm{n-1}}R}}$ is placed between $\mathcal{P_{L^\mathrm{n}R}}$ and $\mathcal{P_{L^\mathrm{n-1}R}}$ i.e., in the molecular region. This shows that the patterns are not arranged monotonically with respect to their cardinalities.
\end{note}
\begin{Lemma}\label{repetition}
 No molecular pattern is a repetition of any single atomic pattern. 
\end{Lemma}

\begin{IEEEproof}
Consider an orbit whose pattern is a repetition of one atomic pattern, say $(\mathcal{L^{\mathrm{n}}R})^k$. From equation \eqref{basic}, one can find a relation between $x_0$ and $x_{n+1}$ given by $x_{n+1} = a^nbx_0 + b\mu S^{a}_{n-1} + \mu - l$ which represents an affine relation between $x_0$ and $x_{n+1}$. Note that the relation between $x_{n+1}$ and $x_{2n+2}$ is exactly the same as the relation between $x_0$ and $x_{n+1}$. In general, the relation between $x_{(i-1)(n+1)}$ and $x_{i(n+1)}$ is exactly the same as the affine relation between $x_0$ and $x_{n+1}$. Since we are assuming that the pattern is $(\mathcal{L^{\mathrm{n}}R})^k$, therefore $x_{k(n+1)}=x_{0}$. Since there is the same affine relation between $x_{(i-1)(n+1)}$ and $x_{i(n+1)}$ for $i = 1, 2, \ldots k$, therefore one can conclude that $x_{i(n+1)}= x_0$ for all $i = 1, 2, \ldots, k$. Thus, the orbit is really an orbit with the pattern \ml. 
\end{IEEEproof}

Putting the last two lemmas together, one can conclude that
\begin{Lemma}\label{only2}
 Every molecular pattern is a combination of exactly two atomic patterns of succesive cardinality. 
\end{Lemma}

\par
We know that molecular patterns are a combination of only two atomic patterns with successive cardinality. We now show that these possible combinations are restricted. In order to do this, we generalize the map given by equation \eqref{basic}  by replacing the symbols  \cL and \cR with the atomic blocks $\mathcal{L^{\mathrm{n}}R}$ and $\mathcal{L^{\mathrm{n-1}}R}$ -- a trick that we have already used in the proof of Lemma~\ref{repetition}. 
\par
Assume $\mu\in\big(\frac{a^{n-1}}{a^{n-1}b+S_{n-1}^a},\,\frac{a^{n-1}}{S_{n-1}^a}\big]$. Therefore, by the arguments stated above, the possible patterns are combinations of $\mathcal{L^{\mathrm{n}}R}$ and $\mathcal{L^{\mathrm{n-1}}R}$. Let us denote $\mathcal{L^{\mathrm{n}}R}$ by $\mathcal{L}'$ and $\mathcal{L^{\mathrm{n-1}}R}$ by $\mathcal{R}'$. From equation \eqref{x01}, when $x_0\leq -\frac{\mu S_{n-2}^a}{a^{n-1}}$ then $n$ consecutive \cLs appears before a \cR appears. In other words, at least one $\mathcal{L}'$ appears in the pattern. When $x_0>-\frac{\mu S_{n-2}^a}{a^{n-1}}$, $n$ consecutive \cLs cannot appear. In this case, at most $n-1$ \cLs could appear. Since the value of $\mu$ is restricted, one can in fact say that exactly $n-1$ \cLs would appear thereby guaranteeing at least one $\mathcal{R'}$ in the pattern. Consider $x_{new}=\frac{-\mu S_{n-2}^a}{a^{n-1}}$ which is a border that decides between $\mathcal{L}'$ and $\mathcal{R}'$. We define a new map as:
\begin{equation*}
\tilde{x}_{N+1}= \left\{
                \begin{array}{ccc}
                \tilde{a}\tilde{x}_{N}+\tilde{\mu} & for & \tilde{x}_{N} \leq x_{new}\\
                \tilde{b}\tilde{x}_{N}+\tilde{\mu} + \tilde{l} & for & \tilde{x}_{N} > x_{new}
                \end{array}
        \right.
\end{equation*}
Where $\tilde{a}=ba^n,\, \tilde{b}=ba^{n-1},\, \tilde{\mu}=\mu bS_{n-1}^a+\mu-1$ and $\tilde{l}=-\mu ba^{n-1}$.  By co-ordinate transformation $y_N=\tilde{x}_{N}-x_{new}$, we shift the border to zero. Hence the map equation becomes:
\begin{equation*}
y_{N+1}=\left\{
                \begin{array}{ccc}
                \tilde{a}y_{N}+\tilde{\mu} - x_{new}(1-\tilde{a}) & for & y_{N} \leq 0\\
                \tilde{b}y_{N}+\tilde{\mu} - x_{new}(1-\tilde{b}) + \tilde{l} & for & y_{N} > 0
                \end{array}
        \right.
\end{equation*}
i.e.
\begin{equation}\label{basic2}
y_{N+1}=\left\{
                \begin{array}{ccc}
                \tilde{a}y_{N}+\bar{\mu} & for & y_{N} \leq 0\\
                \tilde{b}y_{N}+\bar{\mu} + \bar{l} & for & y_{N} > 0
                \end{array}
        \right.
\end{equation}
Here, $\bar{\mu}=\tilde{\mu}-x_{new}(1-\tilde{a})$ and $\bar{l}=\tilde{l}+x_{new}(\tilde{b}-\tilde{a})$. In order to obtain orbits, this map should satisfy the condition $-\bar{l}>\bar{\mu}>0$ (see case 5 in the introduction of this paper). We show that this is indeed true. Consider, $-\bar{l}>\bar{\mu}>0$. Substituting for $\bar{l}$ and $\bar{\mu}$ we get,
$-\tilde{l}-x_{new}(\tilde{b}-\tilde{a})>\tilde{\mu}-x_{new}(1-\tilde{a})>0$. Substituting $\tilde{a},\tilde{b},\tilde{\mu},\tilde{l}$ and $x_{new}$ and simplifying we get, $\frac{a^{n-1}}{S_{n-1}^a}>\mu>\frac{a^{n-1}}{a^{n-1}b+S_{n-1}^a}$. This satisfies our earlier assumption about the range of $\mu$.
\par
Now using the Lemma \ref{thm:funda} one can show that consecutive $\mathcal{L'}$s and consecutive $\mathcal{R'}$s cannot appear simultaneously in any pattern. Similarly, using AMAL lemma, we get conditions on $\bar{\mu}$ for appearance of at most/at least $n$ consecutive $\mathcal{L'}$ in the pattern. Thus, one gets atomic and molecular patterns involving $\mathcal{L'}$ and $\mathcal{R'}$. Further, using Lemma \ref{only2}, we can conclude that every molecular pattern involving $\mathcal{L'}$ and $\mathcal{R'}$ is made up by combining only two atomic patterns of successive cardinality. One can then again define a new map to investigate the molecular region of patterns involving two atomic patterns of $\mathcal{L'}$ and $\mathcal{R'}$. Continuing in this way, one would finally arrive at an atomic pattern in terms of the new symbols defined. This fractal-like process makes the present study even more interesting.

\begin{example}
 Consider a pattern $\mathcal{LLRLLRLR\ LLRLLRLR\ LLRLR}$. This pattern corresponds to a  period-$21$ molecular orbit. Let $\mathcal{LLR}$ be denoted by $\mathcal{L}'$ and $\mathcal{LR}$ be denoted by $\mathcal{R}'$. Then the above pattern becomes $\mathcal{L'L'R'L'L'R'L'R'}$. Further now let $\mathcal{L'L'R'} \equiv \mathcal{LLRLLRLR}$  be denoted by $\mathcal{L}''$ and $\mathcal{L'R'} \equiv \mathcal{LLRLR}$ by $\mathcal{R}''$. Hence the above pattern can be written as $\mathcal{L''L''R''}$ which is atomic in the symbols $\mathcal{L''}$ and $\mathcal{R''}$. Therefore this is an admissible pattern. 
\par
Now consider another pattern say $\mathcal{LLRLLRLR\ LLRLR\ LLRLR\ LLRLLRLR}$. It can be represented as $\mathcal{L''R''R''L''}$. This pattern does not correspond to any admissible orbit as it is not atomic or molecular in the new symbols.
 \end{example}

Note that the results above can be put together to obtain an algorithm for generating admissible patterns. Now that we have characterized for all admissible patterns, we turn to the question of finding how many different patterns exist for any given period $n$. 
\begin{thm}
For any $n$, there exist $\phi(n)$ distinct admissible patterns of cardinality $n$.
\end{thm}
\begin{IEEEproof}
Let $\sigma$ be a pattern with $|\sigma| = n$. Further assume that there are $k$ \cRs in $\sigma$. Hence the number of \cLs in $\sigma$ are $n-k$. Assume without loss of generality, that $k \leq n-k$. If $k=1$, then the pattern is $\mathcal{L^{\mathrm{n-1}}R}$. If $k\neq 1$ then suppose $k \divides (n-k)$ i.e. $n-k=qk$. One possibility is all the $k$ atomic blocks are of type $\mathcal{L^{\mathrm{q}}R}$. This pattern is repetition of $\mathcal{L^{\mathrm{q}}R}$ which is not admissible by Lemma \ref{repetition}.
\par
Another possibility is that at least one atomic block has more than $q$ \cLs. This would force some other atomic block to have less than $q$ \cLs. By Lemma \ref{only2}, each molecular orbit has only two types of atomic blocks with successive cardinality and therefore such cases are not possible.
\par
Now suppose, $k\notdivides (n-k)$ i.e. $n-k=qk+p$. Using Lemma \ref{only2} we conclude that there are $p$ atomic blocks of type $\mathcal{L^{\mathrm{q+1}}R}$ and $k-p$ atomic blocks of type $\mathcal{L^{\mathrm{q}}R}$ as $n-k=q(k-p)+(q+1)p$. Denoting $\mathcal{L^{\mathrm{q}}R}$ by $\mathcal{L}'$ and $\mathcal{L^{\mathrm{q+1}}R}$ by $\mathcal{R}'$, we are now back to the original problem, with $|\sigma'|=k$, with $p$ $\mathcal{R}'$s and $k-p$ $\mathcal{L}'$s in $\sigma'$. Now we set $k$ as new $n$ and $min\{p, k-p\}$ as the new $k$.
\par
 This process is repeated until $p=1$ or $p=k-1$. This is only possible if the original $n$ and $k$ were co-prime. Thus the number of \cLs and \cRs that appear in a period $n$ orbit have to be co-prime to $n$. Since there are $\phi(n)$ numbers co-prime to $n$, there are $\phi(n)$ distinct admissible patterns.
\end{IEEEproof}

The proof above in fact gives us an algorithm of generating admissible patterns of any given period $n$. We demonstrate this with an example:

\begin{example}
Suppose we need to generate all admissible patterns for $n = 18$. From the theorem above, we know there are $\phi(18) = 6$ distinct admissible patterns. Let us find these admissible patterns. The numbers co-prime to $18$ are $1,5,7,11,13,17$ respectively. Thus the $6$ distinct patterns would have $1,5,7,11,13$ and $17$ \cRs in them. The patterns corresponding to $1$ and $17$ are the \cL-atomic and \cR-atomic patterns respectively. 

Consider the case of $5$ \cRs. Then the pattern must contain $13$ \cLs. As $13 = 2 \times 5 + 3$, we can conclude that there must be $3$ copies of $\mathcal{LLLR}$ and two copies of $\mathcal{LLR}$ in the pattern. Now following the proof, we look at patterns of length $5$ having two $\mathcal{R'} = \mathcal{LLR}$ and three $\mathcal{L'} = \mathcal{LLLR}$. Again, since $3 = 1 \times 2 + 1$, we conclude that there must be one pattern of $\mathcal{L'L'R'} = \mathcal{LLLR LLLR LLR}$ and one pattern of $\mathcal{L'R'} = \mathcal{LLLR LLR}$. Thus the pattern corresponding to $5$ is $\mathcal{LLLR LLLR LLR LLLR LLR}$. The case of $13$ \cRs is obtained from the this pattern by interchanging \cL and \cR.

Finally, consider the case of $7$ \cRs and therefore $11$ \cLs. As $11 = 1 \times 7 + 4$, there should be $4$ copies of $\mathcal{LLR}$ and three copies of $\mathcal{LR}$. Now looking at patterns of length $7$ with $4$ $\mathcal{L'} = \mathcal{LLR}$ and three $\mathcal{R'} = \mathcal{LR}$, we have $4 = 1 \times 3 + 1$ and so there should be one copy of $\mathcal{L'L'R'} = \mathcal{LLR LLR LR}$ and two copies of $\mathcal{L'R'}= \mathcal{LLR LR}$. Thus, the final pattern is $\mathcal{LLR LR LLR LR LLR LLR LR}$.
\end{example}

Given a pattern $\sigma$ with $|\sigma| = n$, let us assume that the first symbol in the pattern stands for the point $x_0$. Then one can evaluate $x_n$ and by setting $x_n = x_0$, one obtains an expression for $x_0$ in terms of the parameters $a,b,\mu, l$.
The value of $x_0$ can then be substituted into the inequalities corresponding to each position (as demonstrated in the Atomic Lemma) to obtain $\mu_1$ and $\mu_2$ such that $\mathcal{P_{\sigma}} = (\mu_1,\mu_2]$. If the period $n$ is very large this method of substitution becomes very cumbersome. Amongst all these inequalities, if one knows the precise location of the inequalities that gives one $\mu_1$ and $\mu_2$, then it saves a lot of work. We now state a lemma that helps us find the precise location of  that \cL and \cR in the pattern where, if one substitutes $x_0$, one gets $\mu_{2}$ and $\mu_{1}$. Observe that every \cL gives an upper limit for $\mu$ and every \cR gives a lower limit for $\mu$. 

Given a pattern $\sigma$ with $|\sigma| = n$, we define the binary sequence $\mathfrak{F}(\sigma)$ by substituting $0$ for $\mathcal{L}$ and $1$ for $\mathcal{R}$. Observe that all cyclic shifts of the binary sequence $\mathfrak{F}(\sigma)$ represent the same admissible pattern. Among all the cyclically shifted binary sequences of $\mathfrak{F}(\sigma)$, the sequence that corresponds to the largest binary number is called \cL-way arranged pattern. Similarly, the cyclically shifted binary sequence that corresponds to the smallest binary number is called \cR-way arranged pattern. Observe that a \cL-way arranged pattern always begins with a \cR and ends with a \cL, whereas a \cR-way pattern always begins with a \cL and ends with a \cR. 



\begin{Lemma}\label{arranged}
 The \cR-way arranged pattern gives the location for determining $\mu_{1}$ and the \cL-way arranged pattern gives the location for determining $\mu_{2}$.
\end{Lemma}
\begin{proof}
 Every inequality $x_{i}\leq 0$ gives an upper bound for $\mu$ whereas every inequality $x_{i}> 0$ gives an lower bound for $\mu$. First let us consider a \cL-atomic pattern. From the Atomic Lemma one knows that for a chain of consecutive \cLs, the upper bound becomes tighter with each subsequent \cL. As a result, the value of the upper bound becomes smaller. For an atomic pattern $\mathcal{L^{\mathrm n}R}$, if we rearrange the symbols with all the \cLs \- following the \-\- \cR, then the last \cL gives the minimum upper bound, i.e., the value of $\mu_2$. This is indeed the \cL-way arrangement for the atomic pattern. Meanwhile, the lower limit for $\mu$ is obtained from the lone \cR in the pattern and the \cR-way arrangement would have this \cR as the last symbol. A similar argument applies for a \cR-atomic pattern.  

Let us now consider a molecular pattern made up of \cL-atomic patterns. Following the fractal like argument that we have used before, this molecular pattern can be recursively rewritten as a pattern of new symbols, until one obtains an atomic pattern in those new symbols. If the new symbols are $\mathcal{L'}$ and $\mathcal{R'}$, then using the first part of this proof, we know that the upper and lower bounds can be found from specific $\mathcal{L'}$-way arrangement and $\mathcal{R'}$-way arrangement. On expanding, these symbols into the original string of \cL and \cR, one can then argue that it is indeed the \cL-way arranged pattern and the \cR-way arranged pattern that defines the positions of the \cL and \cR that gives the tightest upper and lower bounds for $\mu$.
\end{proof}

\begin{example}
Let us consider an example to demonstrate the above lemma. Consider a pattern of the form \newline $\mathcal{RLRLLRLLRLRLLRLRLLRLL}$ with 21 symbols. If we start with $x_0$ in \cR and write out the equations for each $x_i$, then we obtain an expression for $x_0$ in terms of $a,b,\mu,l$ by equating $x_0$ to $x_{21}$. Assuming that we know $a,b,l$, each of the inequalities corresponding to $x_i$s give us an upper or lower bound for $\mu$. 

\begin{subfigures}
\begin{figure}[ht!]
\begin{center}
\input{explain.pstricks}
\caption{\cL-way arranged pattern}
\label{examp}
\end{center}
\end{figure}
\begin{figure}[ht!]
\begin{center}
\input{explain1.pstricks}
\caption{\cR-way arranged pattern}
\label{examp1}
\end{center}
\end{figure}
\end{subfigures}

In Figure \ref{examp}, the pattern is arranged in \cL-way, whereas in Figure \ref{examp1}, the pattern is arranged in \cR-way. From these patterns, one can conclude that the upper bound $\mu_2$ for such a pattern is obtained by considering the 5-th \cL in the pattern $\mathcal{RLRLLRLLRLRLLRLRLLRLL}$, that is, by considering the inequality arising from $x_7$ at position $7$. Meanwhile the lower bound $\mu_1$ for such a pattern is obtained by considering the 7-th \cR that appears in the pattern, that is, the inequality arising from $x_{15}$ at position $15$.

For example, if one assumes $a = 0.85, b = 0.8$ and $l = -1$, then one obtains lower bounds at every position having \cR, that is, at positions $0,2,5,8,10,13,15,18$. These values turn out to be $0.3085, 0.3553, 0.3235, 0.3054, 0.3532, \- 0.3223, 0.3650, 0.3290$ respectively. Thus, the value for $\mu_1 = 0.3650$ which comes from the inequality at position $15$ of the original pattern. The cyclic shift that brings the \cR at position $15$ to the last position is indeed the \cR-way arrangement. Similarly, one obtains upper bounds at positions $1,3,4,6,7,9,11,12,14,16,17,19,20$ -- the values obtained are $0.4270, 0.4671, 0.3880, 0.4399, 0.3658, 0.4243, 0.4652, 0.3865, 0.4388, 0.4753, 0.3946, 0.4445, 0.3696$ respectively. Thus, the value for $\mu_2 = 0.3658$ which comes from the \cL at position $7$. The cyclic shift that brings the \cL at position $7$ to the last position is the \cL-way arrangement.

Thus, the above pattern appears for $\mu$ in the range $(0.3650, 0,3658]$. It was also observed that the inequalities one obtains are exactly the same, no matter which cyclic shift one considers as the original pattern.
\end{example}

\section{Conclusions}
In this paper, we have examined stable periodic orbits of piecewise-smooth systems analytically. Using a model given in literature, we first concluded that stable periodic orbits would appear only for certain values of parameters. We considered the case where the parameters $a,b\in (0,\,1)$ and $l = -1$. With these parameters, it turns out that stable periodic orbits exist for $\mu \in (0,1]$. We have shown several interesting results about these periodic orbits. It was shown that stable orbits of any periodicity exists in such a system. The exact patterns for these periodic orbits were determined. It was shown that all periodic orbits can be thought of as a combination of at most two distinct atomic patterns of successive cardinality. Further, it was shown that given any $n$, there are precisely $\phi(n)$ distinct types (patterns) of periodic orbits with cardinality $n$. We have also given an algorithm of determining the range of $\mu$ where the periodic orbits display a particular pattern.  


\bibliographystyle{IEEEtran}

\bibliography{reference}

\begin{thebibliography}{10}
\providecommand{\url}[1]{#1}
\csname url@samestyle\endcsname
\providecommand{\newblock}{\relax}
\providecommand{\bibinfo}[2]{#2}
\providecommand{\BIBentrySTDinterwordspacing}{\spaceskip=0pt\relax}
\providecommand{\BIBentryALTinterwordstretchfactor}{4}
\providecommand{\BIBentryALTinterwordspacing}{\spaceskip=\fontdimen2\font plus
\BIBentryALTinterwordstretchfactor\fontdimen3\font minus
  \fontdimen4\font\relax}
\providecommand{\BIBforeignlanguage}[2]{{%
\expandafter\ifx\csname l@#1\endcsname\relax
\typeout{** WARNING: IEEEtran.bst: No hyphenation pattern has been}%
\typeout{** loaded for the language `#1'. Using the pattern for}%
\typeout{** the default language instead.}%
\else
\language=\csname l@#1\endcsname
\fi
#2}}
\providecommand{\BIBdecl}{\relax}
\BIBdecl

\bibitem{deane}
J.~H.~B. Deane and D.~C. Hamill, ``Instability, subharmonics and chaos in power
  electronics circuits,'' in \emph{Power Electronics Specialists Conference},
  vol.~1.\hskip 1em plus 0.5em minus 0.4em\relax IEEE, June 1990, pp. 34--42.

\bibitem{tse}
C.~K. Tse, ``Flip bifurcation and chaos in three-state boost switching
  regulators,'' \emph{IEEE Transactions on Circuits and Systems I: Fundamental
  Theory and Applications}, vol.~41, no.~1, pp. 16--23, Jan 1994.

\bibitem{grebogi}
E.~Pavlovskaia, M.~Wiercigroch, and C.~Grebogi, ``Two-dimensional map for
  impact oscillator with drift,'' \emph{Physical Review E}, vol.~70, no.~3, pp.
  362\,011--362\,019, September 2004.

\bibitem{bernardo}
M.~di~Bernardo, C.~Budd, A.~Champneys, and P.~Kowalczyk, \emph{Piecewise-Smooth
  Dynamical Systems: Theory and Applications}, ser. Applied Mathematical
  Sciences.\hskip 1em plus 0.5em minus 0.4em\relax London: Springer, 2008, vol.
  163.

\bibitem{dutta}
P.~Dutta, B.~Routroy, S.~Banerjee, and S.~Alam, ``On the existence of low
  period orbits in n-dimensional piecewise linear discontinuous maps,''
  \emph{Nonlinear Dynamics}, vol.~53, no.~4, pp. 369--380, September 2008.

\bibitem{nusse}
H.~Nusse and J.~Yorke, ``Border-collision bifurcations including period two to
  period three for piecewise smooth systems,'' \emph{Physica D: Nonlinear
  Phenomena}, vol.~57, no.~1, pp. 39--57, June 1992.

\bibitem{yorke}
H.~Nusse, E.~Ott, and J.~Yorke, ``Border collision bifurcations: An explanation
  for observed bifurcation phenomena,'' \emph{Physical Review E}, vol.~49,
  no.~2, pp. 1073--1076, February 1994.

\bibitem{1d}
S.~Banerjee and M.~Karthik, ``Bifurcations in one-dimensional piecewise smooth
  maps: Theory and applications in switching circuits,'' \emph{IEEE
  Transactions on Circuits and Systems I: Fundamental Theory and Applications},
  vol.~47, no.~3, pp. 389--394, March 2000.

\bibitem{border}
P.~Jain and S.~Banerjee, ``Border collision bifurcation in one-dimensional
  discontinuous maps,'' \emph{International Journal of Bifurcation and Chaos},
  vol.~13, no.~11, pp. 3341--3351, November 2003.

\bibitem{2d}
S.~Banerjee, B.~P. Ranjan, and C.~Gebogi, ``Bifurcations in two-dimensional
  piecewise smooth maps-theory and applications in switching circuits,''
  \emph{IEEE Transactions on Circuits and Systems I: Fundamental Theory and
  Applications}, vol.~47, no.~5, pp. 633--643, May 2000.

\bibitem{abed}
M.~A. Hassouneh, E.~H. Abed, and H.~Nusse, ``Robust dangerous border-collision
  bifurcations in piecewise smooth systems,'' \emph{Physical Review Letters},
  vol.~92, no.~7, pp. 702\,011--702\,014, February 2004.

\bibitem{homer}
M.~di~Bernardo, M.~Feigin, S.~Hogan, and M.~Homer, ``Local analysis of
  c-bifurcations in n-dimensional piecewise-smooth dynamical systems,''
  \emph{Chaos, Solitons \& Fractals}, vol.~10, no.~11, pp. 1881--1908, November
  1999.

\bibitem{avrutina}
V.~Avrutin and M.~Schanz, ``Border collision period-doubling phenomena,''
  \emph{Physical Review E}, vol.~70, no.~2, pp. 0\,262\,221--02\,622\,211,
  August 2004.

\bibitem{avrutinb}
------, ``On multi-parametric bifurcation in a scalar piecewise-linear map,''
  \emph{Nonlinearity}, vol.~19, pp. 531--552, January 2006.

\bibitem{avrutinc}
------, ``Multi-parametric bifurcations in a piecewise-linear discontinuous
  map,'' \emph{Nonlinearity}, vol.~19, pp. 1875--1906, January 2006.

\bibitem{aligood}
K.~Alligood, T.~Sauer, and J.~Yorke, \emph{Chaos: An Introduction to Dynamical
  Systems}, ser. Textbooks in Mathematical Sciences.\hskip 1em plus 0.5em minus
  0.4em\relax New York: Springer, 1997.

\end{thebibliography}

%

\end{document}